\documentclass[letterpaper, 10 pt, conference]{ieeeconf}  
\IEEEoverridecommandlockouts                              
\usepackage{graphics} 
\usepackage{graphicx,color}
\usepackage[active]{srcltx}
\usepackage{amsmath} 
\usepackage{amssymb}  
\usepackage{amsmath,amssymb,amsfonts,theorem}
\usepackage{enumerate}
\usepackage{tikz}
\usepackage{tabu}
\usepackage{algpseudocode}

\usepackage{epstopdf}

 \theorembodyfont{\normalfont} 

\newcommand{\diag}{\operatorname{diag}} 

{
\newtheorem{example}{Example}
\newtheorem{remark}{Remark}
}

\newtheorem{definition}{Definition}
\newtheorem{theorem}{Theorem}

\newtheorem{lemma}{Lemma}

\newtheorem{proposition}{Proposition}

\def\QEDclosed{\mbox{\rule[0pt]{1.3ex}{1.3ex}}} 
\def\QEDopen{{\setlength{\fboxsep}{0pt}\setlength{\fboxrule}{0.2pt}\fbox{\rule[0pt]{0pt}{1.3ex}\rule[0pt]{1.3ex}{0pt}}}}
\def\QED{\QEDopen}

\usepackage{tcolorbox}
\usepackage{algorithm}
\usepackage{algpseudocode}

\def\be{\begin{equation}}
\def\ee{\end{equation}}
\def\ba{\begin{array}}
\def\ea{\end{array}}
\def\eqa{\begin{eqnarray}}
\def\eqe{\end{eqnarray}}
\title{\LARGE \bf A Game-theoretic Framework for Security-aware Sensor Placement Problem in Networked Control Systems}
\author{Mohammad Pirani, Ehsan Nekouei, Henrik Sandberg and Karl Henrik Johansson
\thanks{This work is supported by the Knut and Alice Wallenberg Foundation, the Swedish Foundation for Strategic Research and the Swedish Research Council. Authors are with the Department of Automatic Control, KTH Royal Institute of Technology. E-mail: {\tt \{pirani,nekouei,hsan,kallej\}@kth.se}.}%
}

\begin{document}
\maketitle
\thispagestyle{empty}
\pagestyle{empty}
\begin{abstract}
This paper studies the sensor placement problem in a networked control system for improving its security against cyber-physical attacks. The problem is formulated as a zero-sum game between an attacker and a detector. The attacker's decision is to select $f$ nodes of the network to attack whereas the detector's decision is to place $f$ sensors to detect the presence of the attack signals. In our formulation, the attacker minimizes its visibility, defined as the system $L_2$ gain from the attack signals to the deployed sensors' outputs, and the detector maximizes the visibility of the attack signals. The equilibrium strategy of the game determines the optimal locations of the sensors.  The existence of Nash equilibrium for the attacker-detector game is studied when the underlying connectivity graph is a directed or an undirected tree. When the game does not admit a Nash equilibrium, it is shown that the Stackelberg equilibrium of the game,  with the detector as the game leader, can be computed efficiently. Finally, the attacker-detector game is studied in a cooperative adaptive cruise control algorithm for vehicle platooning problem. The existence of Nash equilibrium is investigated for both directed and undirected platoons and the effect of the position of the reference vehicle on the game value is studied. Our results show that, under the optimal sensor placement strategy, an undirected topology provides a higher security level for a networked control system compared with its corresponding directed topology.
\end{abstract}
\section{Introduction}

\subsection{Motivation}
Applications of distributed control systems, ranging from power grids and smart buildings to intelligent transportation systems, have had a considerable growth. In this direction,  the need to do a rigorous research on the control-theoretic approaches to the security of these systems against failures and attacks, considering the physical limitations of the system, is seriously felt  \cite{Stankovic}. Several approaches have been proposed in the literature to tackle this issue
\cite{Fawzi, sandberg1,teixeira2010cyber, pasqualetti, James, brunosin} which are based on the system specifications and the attack strategy. An active line of research in this area is to consider the defense mechanism in the control system as a game between the attacker and the defender and optimize the actions of the defender against possible attack strategies. In this direction, the game objective can be the effect of the attack on the system in which the defender tries to minimize. However, one can use such game-theoretic approaches to increase the visibility and awareness of the attacker's actions, which the defender tries to maximize, and the problem introduced in this paper is of this kind.  To improve such an awareness against cyber-physical attacks, typically a set of monitoring sensors are deployed in the network and their outputs are used to monitor the security status of the system.

In a networked control system, the system designer determines the location of the monitoring sensors (or detectors). However, the security level not only depends on the sensors' locations but also on the nodes selected by the attacker to inject the attack signals. These decisions are made by different entities with conflicting objectives. In this paper, the sensor/attack placement problem is posed as a game between an attacker and a detector and the equilibrium solution of the game is used to determine the location of the sensors. This allows the system designer to anticipate the behavior of the attacker and decides the location of sensors such that the impact of the attacker's decision on the security level is minimized.  

\subsection{Related Work}
There is a vast literature on game-theoretic approaches to the security and resilience of control systems in the past decade \cite{zhubasar}. These approaches vary depending on the structure of the cyber-physical system or the specific type of malicious action acting on the cyber layer. In the earlier approach, at each layer (physical and cyber) a particular game is defined which emerges the concept of {\it games-in-games} that reflects two interconnected games, one in the physical layer and the other in the cyber layer where the payoff of each game affects the result of the other one \cite{Basarpan, Basarz}. In the latter approach (games based on the type of malicious action), depending on the type of the adversarial behavior (active or passive) appropriate game strategy, e.g., Nash or Stackelberg, was discussed \cite{Cedricbasar,EPFL}.  In addition to these approaches, the evolution of some network control systems are modeled as cooperative games \cite{Mardenshamma} and the resilience of these cooperative games to the actions of adversarial agents or communication failures are studied \cite{Brown, Shankar}.  

\subsection{Contributions}
In this paper, we study the sensor placement problem in a leader-follower networked dynamical system\footnote{Leader-follower systems have diverse applications from multi-agent formation control and vehicle platooning \cite{hao2013stability} to opinion dynamics in social networks \cite{Clarkbook}.} for improving its security against cyber-physical attacks. The sensors placement problem is posed as a zero-sum game between a detector and an attacker. The detector's strategy is to place $f$ sensors in $f$ nodes of the network  to maximize the visibility of the attacker's action. The attacker strategy is to select $f$ nodes to inject its attack signal with minimum visibility to the sensors. The objective of each player is defined as the $L_2$ gain of the system from the injected signals to the sensors' outputs. The equilibrium strategy of the detector determines the location of the sensors.      

Our main contributions can be summarized as follows:
\begin{itemize}
  \item We characterize the Nash equilibrium (NE) strategy of the attacker-detector game for $f=1$ when the underlying connectivity graph is a directed/undirected tree. It is shown that this game may not admit a NE for $f>1$, and instead Stackelberg game between the attacker and detector is analyzed when the detector acts as the game leader. A low complexity algorithm for computing the Stackelberg equilibrium of the game is proposed for both directed and undirected trees.
  
    
    \item We apply these results to discuss the security of a cooperative cruise control algorithm in vehicle platoons. More specifically, we discuss the existence and value of Nash equilibrium for the attacker-detector game on both directed and undirected platoons and the role of the position of the leading vehicle on the game value. 
    \end{itemize}

Our results indicate that the value of the attacker-detector game over a directed  tree is at most equal to that over its corresponding undirected  tree. This observation signifies the importance of two-way communication links in improving the security of networked control systems against cyber-physical attacks. Moreover, our results show that the security of a platoon, as a function of its leader location, is maximized when the leader is located at either ends of the platoon for both directed and undirected topologies.

\begin{remark}
Our analytically results are established by deriving a closed-form expression for the system $L_2$ gain of a networked control system,  via graph-theoretic interpretations of its underlying connectivity graph, for both directed and undirected trees.
\end{remark}
\subsection{Notations and Definitions}
\label{sec:not}

We use $\mathcal{G}_u=\{\mathcal{V},\mathcal{E}\}$ to denote an unweighted undirected graph where $\mathcal{V}$ is the set of vertices (or nodes) and $\mathcal{E}$ is the set of undirected edges where  $(v_i,v_j)\in \mathcal{E}$ if an only if there exists an undirected edge between $v_i$ and $v_j$. Moreover $\mathcal{G}_d=\{\mathcal{V},\mathcal{E}\}$  denotes an unweighted directed graph where $\mathcal{E}$ is the set of directed edges, i.e., $(v_i,v_j)\in \mathcal{E}$ if an only if there exists a directed edge from $v_i$ to $v_j$.  In this paper, directed graphs only have unidirectional edges, i.e., if there exists a direct edge from $v_i$ to $v_j$ in $\mathcal{G}_d$, then there is no direct edge from $v_j$ to $v_i$. Let $|\mathcal{V}|=n$ and define the adjacency matrix for $\mathcal{G}_d$, denoted by $A_{n\times n}$, to be a binary matrix  where  $A_{ij}=1$ if and only if there is an edge from $v_j$ to $v_i$ in $\mathcal{G}_d$ (the adjacency matrix will be a symmetric matrix when the graph is undirected). The {\it neighbors} of vertex $v_i \in \mathcal{V}$ in the graph $\mathcal{G}_d$ are denoted by the set $\mathcal{N}_i = \{v_j \in \mathcal{V}~|~(v_j, v_i) \in \mathcal{E}\}$. We define the in-degree (or just degree for undirected network) for node $v_i$ as $d_i=\sum_{v_j\in \mathcal{N}_i} A_{ij}$.  The Laplacian matrix of an undirected  graph is denoted by $L = D - A$, where $D = \diag (d_1, d_2, ..., d_n)$. We use $\mathbf{e}_i$ to indicate the $i$-th vector of the canonical basis.

\subsection{Organization Of The Paper}
The structure of the paper is as follows. In Section \ref{sec:defini} we introduce the mathematical formulation of the attacker-detector game in a leader-follower consensus dynamics. We then analyze  equlibriums for this game when the underlying network is an undirected tree, Section \ref{sec:undgame}, or a directed tree, Section \ref{sec:dirgame}.  Then we apply these results to a vehicle platooning scenario in Section \ref{sec:platoongame}. Section \ref{sec:conclusion} concludes the paper.

 \section{Problem Definition} \label{sec:defini}
In this section, we propose  a game-theoretic approach to the security of a leader-follower networked control system. Consider a connected network $\mathcal{G}=\{\mathcal{V},\mathcal{E}\}$ comprised  of a leader (or reference) agent, denoted by $v_{\ell}$, and a set of follower agents denoted by $F$ .  The state of each follower agent $v_j\in {F}$ evolves based on the interactions  with its neighbors as
\begin{align}
\dot{x}_{j}(t)&=\sum_{v_i\in \mathcal{N}_j}(x_i(t)-x_j(t)).
\label{eqn:partial}
\end{align}
The state of the leader (which should be tracked by the followers)  evolves with an exogenous reference signal $u(t)$ as
\begin{equation}
{x}_{\ell}(t) = u(t).
\label{eqn:fully}
\end{equation}
If the graph is connected, the  states of the follower agents will track the reference signal $u(t)$ \cite{ArxiveRobutness}. We assume without loss of generality that the leader agent is placed last in the ordering of the agents. The updating rule of each agent is prone to an intrusion (or attack).\footnote{We assume that the leader is not affected by the attacks.} More particularly, there exists an attacker which chooses  $f$ nodes in the network to inject the attack signals to.\footnote{The number of nodes under attack in practice is unknown and we can assume $f$ is an upper bound for the number of attacks.} Hence, if  the dynamics of follower $v_j$ is influenced by an attacker, it will be in the following form 
  \begin{equation}\label{eqn:parewftial}
\dot{x}_{j}(t)=\sum_{v_i\in \mathcal{N}_j}(x_i(t)-x_j(t))+w_j(t),
 \end{equation}
 where $w_j(t)>0$ represents the attack signal.  To detect the presence of the attackers, a defender deploys $f$ dedicated sensors (or detectors)  at $f$ specific follower nodes, denoted by $\mathcal{D}$. Thus we have 
 \begin{equation}
y_i(t)=x_i(t) \quad {\rm if} \quad v_i\in \mathcal{D}.
 \end{equation}
 where $y_i(t)$ is the output of the sensor (detector) deployed at follower $v_i$. 
Aggregating the states of all followers  into a vector $x_{F}(t) \in \mathbb{R}^{n-1}$, and aggregating the attack signals to $\boldsymbol w(t)$,  equations \eqref{eqn:partial} and \eqref{eqn:fully} along with the output measurement yield the following dynamics 
\begin{align}
\begin{bmatrix}
      \dot{\boldsymbol x}_{F}(t)          \\[0.3em]
       \dot{ x}_{\ell}(t) 
     \end{bmatrix}&=-\underbrace{\begin{bmatrix}
       L_g & L_{12}           \\[0.3em]
       \mathbf{0} & \mathbf{0}           
     \end{bmatrix}}_L\begin{bmatrix}
      {\boldsymbol x}_F(t)          \\[0.3em]
       { x}_{\ell}(t) 
     \end{bmatrix}+\begin{bmatrix}
      0          \\[0.3em]
       1
     \end{bmatrix} \dot{u}(t)+\begin{bmatrix}
      B          \\[0.3em]
       0
     \end{bmatrix}\boldsymbol w(t),\nonumber\\
     \boldsymbol y(t)&=C\boldsymbol x_F(t),
\label{eqn:mat}
\end{align}
where $L_{g}$ is called the grounded Laplacian matrix (formed by removing the row and the column corresponding to the leader), the submatrix $L_{12}$ of the graph Laplacian captures the influence of the leader on the followers, $B_{n\times f}=[\textbf{e}_1, \textbf{e}_2, ..., \textbf{e}_f]$, and $C_{f\times n}=[\textbf{e}_1^T; \textbf{e}_2^T; ...; \textbf{e}_f^T]$. In  words, for matrices $B$ and $C$ which specify the actions of the attacker and the detector, respectively, there is a single 1 in the $i$-th row (column) of matrix $B$ ($C$) if the $i$-th node is under attack (has a sensor).\footnote{ Note that the action of the attacker is to choose matrix $B$ and the value of the attack signal $\boldsymbol w (t)$ is not a decision variable.} We assume that there exists at least one attack to the system, i.e., $f\geq 1$. When the graph $\mathcal{G}$ is connected, $L_g$ is nonsingular and $L_g^{-1}$ is nonnegative elementwise \cite{PiraniSundaramArxiv}.  An example of the dynamics in \eqref{eqn:mat} is shown in Fig.~\ref{fig:sstdnj}. 
\begin{figure}[t!]
\centering
\includegraphics[scale=.35]{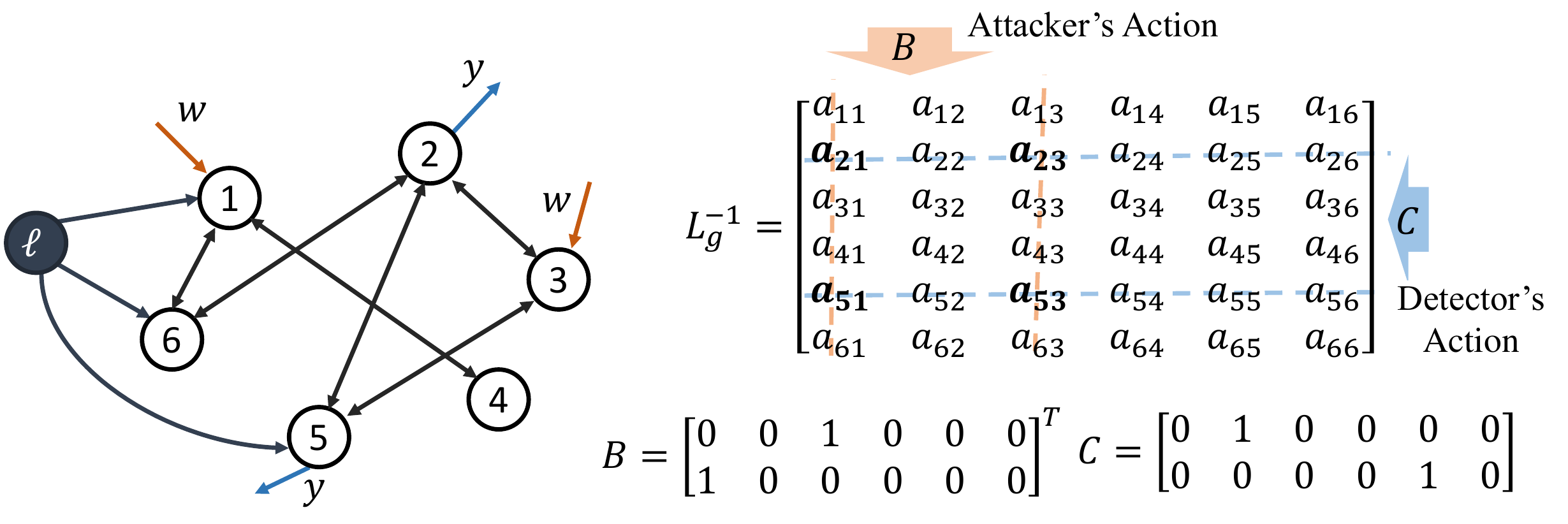}
\caption{An Example of an attacker-detector game with $f=2$.}
\label{fig:sstdnj}
\end{figure}
In this example, a $2\times 2$ submatrix is chosen by the attacker  and the detector from $L_g^{-1}$ which is shown in bold.  Based on \eqref{eqn:mat}, the dynamics of the follower agents are given by 
\begin{align}
\dot{\boldsymbol x}_{F}(t) &= -{L}_g\boldsymbol x_{F}(t) + L_{12} u(t)+B\boldsymbol w(t),
\nonumber\\
 \boldsymbol y(t)&=C\boldsymbol x_F(t).
\label{eqn:partial4}
\end{align}
 The following theorem characterizes the system $L_2$ gain from the attack signal to the output measurement of \eqref{eqn:partial4}.
 \begin{theorem}[\cite{ArxiveRobutness}]
 The system $L_2$ gain from the attack signal to the output measurement of \eqref{eqn:partial4} is given by 
 \begin{equation}\label{eqn:l2gain}
   \sup_{||\boldsymbol w||_2\neq 0} \frac{||\boldsymbol y||_2}{||\boldsymbol w||_2}=\sigma_{\max}(G(0))=  \sigma_{\max}(CL_g^{-1}B)
 \end{equation}
 where  $\sigma_{\max}$ is the maximum singular value of matrix $G(0)$ and the $L_2$ norm of signal $\boldsymbol u$ is $||\boldsymbol u||_2^2\triangleq \int_0^{\infty}\boldsymbol u^T\boldsymbol udt$.
 \label{thm:metzlertheroem}
 \end{theorem}

 Based on Theorem \ref{thm:metzlertheroem}, the attacker-detector game is defined as follows: 

\begin{tcolorbox}
\textbf{Attacker-Detector Game:} We model the interaction between the attacker and the detector as a zero-sum security game. In this game, the attacker's decision is the location of each attack signal, i.e., the matrix $B$ and the detector's decision is the location of sensors, i.e., matrix $C$. The attacker's objective is to reduce the visibility of the attack signal at the output by minimizing the  system $L_2$ gain \eqref{eqn:l2gain} whereas the detector's objective is to increase the visibility of the attack signal at the output by maximizing the $L_2$ gain. 
\end{tcolorbox}

Based on the definition above, the attacker-detector game is a matrix game. For the case  $f=1$, the well-known matrix game is formed with the payoff matrix equal to $[L_g^{-1}]_{ij}\geq 0$. When $f>1$, the payoff will be the largest singular value of the nonnegative matrix $CL_g^{-1}B$. 

\begin{remark}
 The reason of choosing $L_2$ gain \eqref{eqn:l2gain} as the game payoff is that the attacker is willing to be as stealthy as possible by minimizing the largest system norm (worst case gain from its perspective) over all frequencies. Having this attitude from the attacker, the detector tries to maximize this payoff. 
 \end{remark}

Next lemma states a property of the non-negative matrices which is helpful in the equilibrium analysis of the attacker-detector game. 

\begin{lemma}[\cite{vanmeighem}]\label{lem:vanmeigh}
The largest singular value of a nonnegative matrix $M$  is a non-decreasing function of its entries. Moreover, if $M$ is  irreducible, its singular value is strictly increasing with its entries. 
\end{lemma}

 \section{Equilibrium Analysis of the Attacker-Detector Game: Undirected Trees} \label{sec:undgame}
 In this section, we analyze the equilibrium of the attacker-detector game on undirected trees. We first provide an explicit characterization of $L_{g}^{-1}$, for undirected trees,  in terms of the properties of its underlying connectivity graph. This result is helpful in our equilibrium analysis and allows us to investigate the game value. The proof of this result is presented in Appendix \ref{sec:lem2}. 

\begin{lemma}\label{lem:treehinf}
Suppose that $\mathcal{G}_u$ is an undirected tree and  let $\mathcal{P}_{i\ell}$ be the set of nodes involved in the (unique) path from the leader node $v_{\ell}$ to $v_i$ (including $v_i$). Then  we have
\begin{equation}
   [L_g^{-1}]_{ij}= |\mathcal{P}_{i\ell} \cap \mathcal{P}_{j\ell}|.
\end{equation}
\end{lemma}
According to this lemma, the $(i,j)$th element of $L_g^{-1}$ is equal to the number of common edges between the path from the leader to the node $v_i$ and the path  from the leader to the node $v_j$. As an example, $|\mathcal{P}_{3\ell} \cap \mathcal{P}_{6\ell}|=1$ for nodes $3$ and $6$ and $|\mathcal{P}_{3\ell} \cap \mathcal{P}_{4\ell}|=2$ for nodes $3$ and $4$ in Fig. \ref{fig:1} (a).

\subsection{Equilibrium Analysis: $f=1$}

In the single attacker-detector case, i.e., $B=\mathbf{e}_i$ and $C=\mathbf{e}_j^T$ for some $1 \leq i,j\leq n$, the system $L_2$ gain will become
\begin{align}\label{eqn:soigoi}
 \mathbf{e}_j^TL_{g}^{-1}\mathbf{e}_i=[L_{g}^{-1}]_{ij},
\end{align}
where $[L_{g}^{-1}]_{ij}$ is the $ij$-th  element of $L_{g}^{-1}$. 

The following theorem establishes the existence of NE for the attacker-detector game with $f=1$. 
\begin{theorem}\label{thm:ssodjvn}
Let $\mathcal{G}_u$ be an undirected  tree and $v_{\ell}$ be the leader node. Then, for $f=1$, 
\begin{itemize}
    \item[(i)] The attacker-detector game admits at least one NE if $v_{\ell}$ is not a cut vertex and the game value is 1 for all NE in this case. 
    
    \item[(ii)] The game does not admit any NE if $v_{\ell}$ is a cut vertex.
\end{itemize}
\end{theorem}
\begin{proof}
For part (i), the NE belongs to the case where the attacker (the minimizer) chooses the column corresponding to the leader's neighbor.  According to Lemma \ref{lem:treehinf} since all elements of this column are all 1, then, regardless of the actions of the detector, the game payoff will be  1.  Moreover, if the attacker chooses a node other than the leader's neighbor, the payoff will be at least 1.  Hence, not the attacker, nor the detector have an incentive to change their strategy.  For part (ii), if the leader is removed, the  graph will be splitted into two parts and the resulting grounded Laplacian matrix, and consequently $L_g^{-1}$, become block diagonalized. Assume that a NE exists in this case and let $(i^*,j^*)$ denote the equilibrium strategies of the attacker and detector. Thus, we should have
 \begin{equation}\label{eqn:sdgjb}
[L_g^{-1}]_{ij^*}\leq [L_g^{-1}]_{i^*j^*}\leq [L_g^{-1}]_{i^*j}
\end{equation}
for all $i\neq i^*$ and $j\neq j^*$. 
If element $[L_g^{-1}]_{i^*j^*}$ is in one of the zero blocks, as shown in Fig.~\ref{fig:1} (b), then the left inequality will be violated and if it is in one of the nonzero blocks, the right inequality will be violated. 
\end{proof}
\subsection{Equilibrium Analysis: $f>1$}

For $f>1$, the attacker-detector game deos not admit a Nash equilibrium in general as shown in the following example.
\begin{example}\label{exp:aofn}
In Fig.~\ref{fig:1} (c) for the case of $f=2$, it is clear, according to Lemma \ref{lem:vanmeigh}, that one of the choices of the attacker is node 1. Then for the second choice, both attacker and detector should choose from the blocks of all 1 or the red blocks. Thus, similar to the proof of part (ii) of Theorem \ref{thm:ssodjvn}, there would be no NE for the game. In Fig.~\ref{fig:1} (d) there exists a NE for $f=2$.  
\end{example}

\begin{figure}[t!]
\centering
\includegraphics[scale=.3]{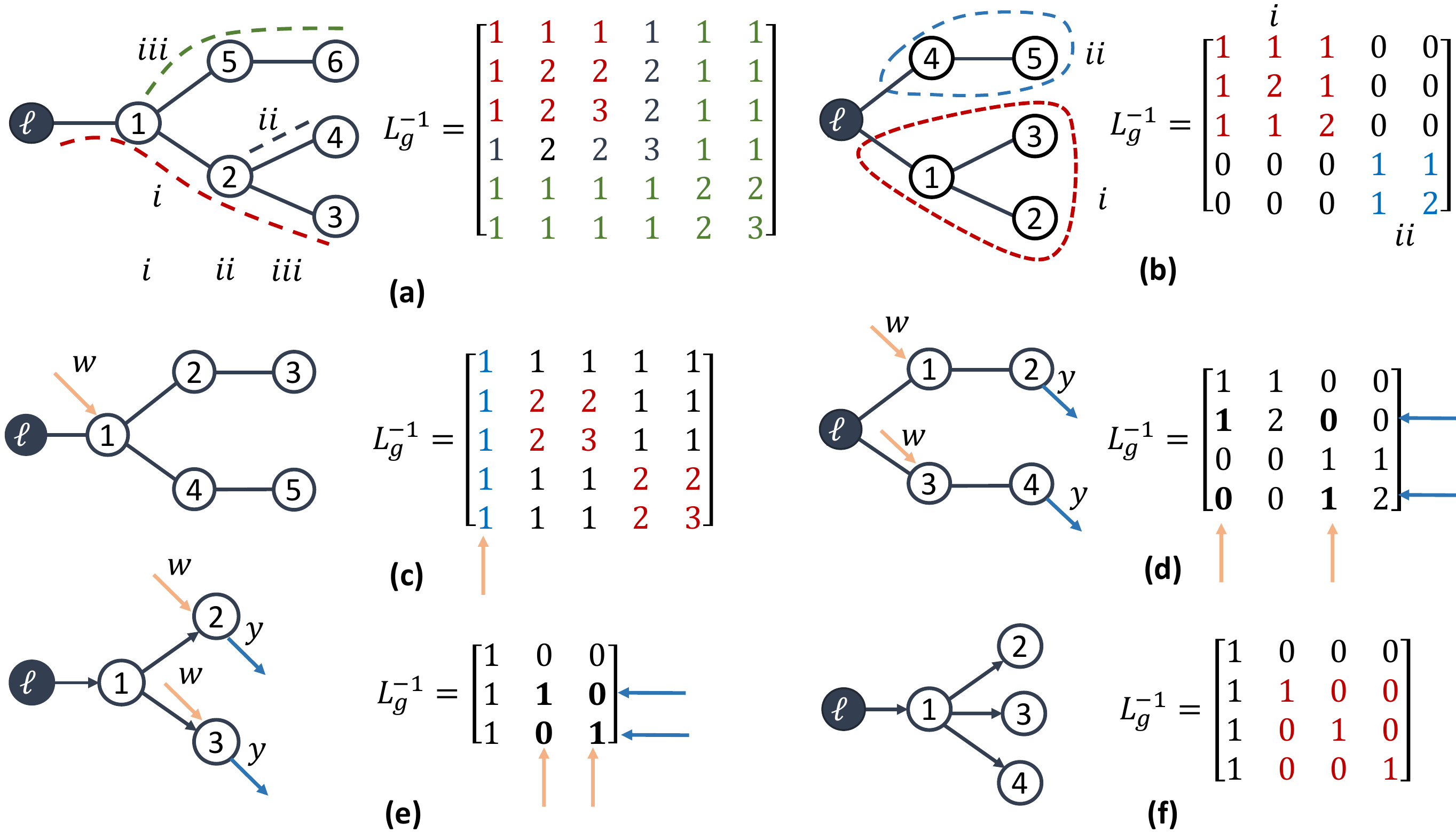}
\caption{(a) An undirected tree with its three paths to leader $v_{\ell}$, (b) An undirected tree where $v_{\ell}$ is a cut vertex, (c) An undirected tree which does not admit NE for $f=2$, (d) n undirected tree with a NE for $f=2$, (e) A directed tree with NE for $f=2$, (f) A directed tree which does not admit NE for $f=2$.}
\label{fig:1}
\end{figure}

\subsection{Stackelberg Game Approach $f>1$}

According to the Example \ref{exp:aofn}, a NE may not exist for general trees. More formally, the following equality does not hold in general
$$\min_B\max_C \sigma_{\max}(CL_g^{-1}B) = \max_C\min_B \sigma_{\max}(CL_g^{-1}B).$$
In this case, we study the Stackelberg equilibrium strategy of the attacker-detector game when the detector acts as the game leader and the attacker acts as the follower. In the Stackelberg game formulation, the leader solves the following optimization problem 
\begin{equation}\label{eqn:stackd}
J^*(C)=\max_C\sigma_{\max}\left(CL_g^{-1}B^*(C)\right).
\end{equation}
where $B^*(C)$ is the best response of the attacker when the strategy of the detector is $C$, i.e., $B^*(C)$ is the solution of the following optimization problem 
\begin{equation}\label{eqn:stack}
B^*(C)=\arg\min_B\sigma_{\max}\left(CL_g^{-1}B\right).
\end{equation}
In particular, for a given strategy of the detector, i.e., $C$, the attacker finds its best response strategy to the detector's decision, which is given by $\min_B CL_g^{-1}B$. Then, the detector optimizes its decision based on all possible best response strategies of the attacker. Unlike the NE,  a Stackelberg game always admits an equilibrium strategy.

 In general, the computation complexity of solving \eqref{eqn:stackd} is $O\left(\binom{n}{f}^2\right)$. That is, the attacker needs to solve \eqref{eqn:stack} for all possible choice of $f$ victim nodes. Then, the detector selects the sensor placement strategy which maximizes \eqref{eqn:stack}. However, based on properties of the grounded Laplacian matrix, we propose an algorithm for finding the Stackelberg equilibrium with much less computational cost.   This algorithm, in a nutshell, is that both attacker and detector identify all $m$ leader-rooted paths\footnote{A leader-rooted path in a tree is a unique path starting from the leader and ends at a node with degree 1.} in $\mathcal{G}$. Then for each partition of $f$ into $m$ nonnegative values $f_1, f_2, ..., f_m$, the detector (attacker) places $f_i$ sensors (attacks) to $f_i$ farthest (closest) nodes to the leader in the $i$-th leader rooted path (i.e., there is no computational cost for the placement of attacks and detectors for a given partitioning).
 The proposed algorithm for solving the Stackelberg game is shown in Algorithm 1.  As it will be shown in Theorem \ref{thm:costfew}, the complexity of this algorithm does not scale with the network size. 
\begin{algorithm}\label{Alg: A1}
\caption{Stackelberg Attacker-Detector Game on Undirected Trees.}
\texttt{// Inputs:}  $\mathcal{G}(\mathcal{V},\mathcal{E})$, $f$
\label{alg: harrary}
\begin{algorithmic}
\State $J^*\leftarrow \mathbf{0}_{\mathcal{S}_{f,m}}$, where $\mathcal{S}_{f,m}$ is the number of solutions of \eqref{eqn:lkdn}.
\For{$i=1:\mathcal{S}_{f,m}$}
\For{$j=1:\mathcal{S}_{f,m}$}

$B^*(C_i)=\arg\min_{B_j}\sigma_{\max}(C_iL_g^{-1}B_j)$
\EndFor

$J^*_i=\sigma_{\max}\left(C_iL_g^{-1}B^*(C_i)\right)$
\EndFor

\texttt{// Output:}  $C^*=\arg\max_{C_i}J^*_i$
\end{algorithmic}
\end{algorithm}	

\begin{theorem}\label{thm:costfew}
Consider the Stackelberg attacker-detector game, with the detector as the game leader, over the connected tree $\mathcal{G}_u$ with leader node $v_{\ell}$ and $m$ leader rooted paths.  Then, Algorithm 1 finds the Stackelberg equilibrium of the game. Moreover, its  computational complexity is $O\left(\mathcal{S}_{f,m}^2\right)$, where $\mathcal{S}_{f,m}$ is the number of constrained partitions of $f$ into $m$ nonnegative integers, i.e., the integer solutions of
\begin{equation}\label{eqn:lkdn}
f=\sum_{i=1}^mf_i, \quad 0\leq f_i\leq \ell_i,
\end{equation}
where $\ell_i$ is the length of the  $i$-th leader rooted path.
\end{theorem}
\begin{proof}
Without loss of generality, we label the nodes in a tree in the following form. We start labeling the nodes a leader rooted path from the leader neighbor, node 1, to a leaf, called $\ell_1$. Then we continue from another leader rooted path which has maximum sharing nodes with the previous leader rooted path can label that from $\ell_1 +1$ to the leaf called $\ell_2$. We continue labeling until all nodes are labeled and the leaf of the last leader rooted path is called $\ell_m$. For the proof, it is sufficient to show that $f_i$ attackers ($f_i$ detectors)  have to be placed in the first (last) $f_i$ columns ($f_i$  rows) of partition $i$. We prove this by contradiction for placing the attack signals and the detector case it follows the same discussion. Let's denote $C_i$ to be the set of columns from $\ell_i+1$ to $\ell_{i+1}$. By contradiction, suppose there exists at least one column $C_i^j$ of $C_i$ where $j<f_i$ which is not chosen by an attacker. Since in this case there exists another column $C_i^h$, $h>j$ which is chosen by an attacker and, as a consequence of Lemma \ref{lem:treehinf}, each elements of $C_i^j$ is smaller than  or equal to $C_i^h$, this contradicts the optimal strategy of the attacker and the proof is complete.
\end{proof}

 \section{Equilibrium Analysis of the Attacker-Detector Game: Directed Trees}\label{sec:dirgame}

In this section, we investigate the existence of equilibrium for the attacker-detector game, when the underlying network is a directed tree. We present the following assumption.
\newline
\textbf{Assumption 1}: In  directed tree $\mathcal{G}_d$  each follower $v_i$ can be reached through a directed path from  leader $v_{\ell}$.  

 Similar to Lemma \ref{lem:treehinf}, we derive a closed-from expression for the inverse of grounded Laplacian matrix $L_g^{-1}$ for the directed case. This result is presented in the next lemma and its proof is presented in Appendix \ref{sec:lem3}. 

\begin{lemma}\label{lem:dirtreehinf}
Suppose that $\mathcal{G}_d$ is a directed tree  with the leader node $v_{\ell}$ satisfying Assumption 1. Then, the entries of the matrix $L_g^{-1}$ are given by 
\begin{equation}
   [L_g^{-1}]_{ij}=
  \begin{cases}
    1       & \quad  \text{if there is a directed path from $j$ to $i$},\\
   0  & \quad \text{if there is no directed path from $j$ to $i$}.
  \end{cases}
\end{equation}
\end{lemma}

\subsection{Equilibrium Analysis: $f=1$}

The following theorem discusses the existence of NE for the attacker-detector game with dynamics \eqref{eqn:partial4} on directed trees when $f$ is equal to $1$.  

\begin{theorem}
Suppose that $\mathcal{G}_d$ is a directed tree with the leader node $v_{\ell}$ satisfying Assumption 1. Then, the attacker-detector game does not accept a NE $f=1$ except when $\mathcal{G}_d$ is a directed path.
\end{theorem}
\begin{proof}
We know that $L_g^{-1}$ is a lower triangular matrix with diagonal elements equal to 1,  due to the fact that the diagonal elements of $L_g^{-1}$ in this case are the inverses of the in-degrees of the nodes and the in-degree of each node is 1. Thus, there exists at least one element 1 in each row and column  of $L_g^{-1}$. Moreover, based on Lemma \ref{lem:dirtreehinf}, $L_g^{-1}$ is a binary matrix. A NE state  should satisfy \eqref{eqn:sdgjb}. If $[L_g^{-1}]_{i^*j^*}=0$ then the left inequality in \eqref{eqn:sdgjb} will be violated and if $[L_g^{-1}]_{i^*j^*}=1$ the right inequality is violated unless the elements in the $i^*$-th row are all 1. This means, based on Lemma \ref{lem:dirtreehinf}, that there must be a directed path from any node to node $v_{j^*}$ and this means that $v_{j^*}$ is at the end of a directed path graph which yields the result. 
\end{proof}

\subsection{Equilibrium Analysis: $f>1$}

Similar to the case of undirected trees, for directed trees when $f>1$ we may or may not have NE in general, as shown in the following example.
\begin{example}
It can be easily checked that the attacker-detector game over the  directed tree with $f=2$ shown in Fig.~\ref{fig:1} (e) has a NE, whereas it does not admit a NE over the graph in Fig.~\ref{fig:1} (f). It is because of the fact that the attackers chooses its target nodes from nodes $2,3,4$, since the first column of $L_g^{-1}$ is all 1 and choosing it will result in a larger payoff (Lemma \ref{lem:vanmeigh}). As the detector tries to maximize the payoff, it will also choose from these three nodes. Thus the  corresponding block in $L_g^{-1}$ is an identity matrix which does not admit a NE. 
\end{example}

\subsection{Stackelberg Game Approach $f>1$}

 Although for many directed trees there is no NE, similar to the case of undirected trees, we can show that performing the Stackelberg max-min game does not cost much computational effort.

\begin{theorem}
Let $\mathcal{G}_d$ be a directed tree with leader node $v_{\ell}$ and $m$ leader rooted paths satisfying Assumption 1. Then the objective function \eqref{eqn:stack} can be solved within $\mathcal{S}_{f,m}^2$ iterations, where $\mathcal{S}_{f,m}$ is the number of constrained partitions of $f$ into $m$ nonnegative integers, i.e., the integer solutions of \eqref{eqn:lkdn}.
\end{theorem}

\begin{proof}
 The procedure of the proof is similar to that of Theorem \ref{thm:costfew}. However, in this case the attackers or detectors selected for each partition $i$, called $f_i$, are placed in the end of the partition.  
\end{proof}

 The following theorem compares the value of the attacker-detector game when the underlying networks are directed and undirected trees. The proof is straightforward based on Lemmas \ref{lem:treehinf} and \ref{lem:dirtreehinf}
as well as the monotonicity of the largest singular value, mentioned in Lemma \ref{lem:vanmeigh}. 

\begin{theorem}\label{thm:comparison}
Let $\mathcal{G}_d$ be a directed tree with leader node $v_{\ell}$ and $\mathcal{G}_u$ be its corresponding undirected graph (by removing directions from the edges). Let the value of the Stackelberg game between $f$ attackers and detectors on $\mathcal{G}_d$ and $\mathcal{G}_u$ be $J_d$ and $J_u$, respectively. Then we have $J_d\leq J_u$. 
\end{theorem}

\section{Application to Secure platooning}\label{sec:platoongame}

In this section, we consider  a network of connected vehicles and  study the attacker-detector game for its {\it cooperative adaptive cruise control dynamics}. In this setting, the objective for each follower vehicle  is to track a reference velocity (computed by the leader vehicle to optimize a certain objective, e.g., fuel consumption \cite{van2015fuel}) while the vehicle remains in a safe distance from its neighboring vehicles. We use the results of the previous sections to study the existence of NE for the attacker-detector game in the platooning application and its corresponding game value. Note that the underlying connectivity graph is a line in the platooning application. This property allows us to provide a more detailed equilibrium analysis of the attacker-detector game for the vehicle platooning application compared with the attacker-detector game over trees.

Consider a connected network of $n$ vehicles. The position and longitudinal velocity of each vehicle $v_i$ is denoted by scalars $p_i(t)$ and $u_i(t)$, respectively. Each vehicle $v_i$ is able to communicate with its neighbor vehicles and transfer its kinematic parameters, e.g., velocity. \footnote{Transmitting vehicle's states such as velocity is common in standard short-range vehicular communications \cite{IEEE1609}.} 
  The desired vehicle formation will be determined by specific constant inter-vehicular distances. Let $\Delta_{ij}$ denote the desired distance between vehicles $v_i$ and $v_j$. The desired vehicle formation and velocity tracking is schematically shown in Fig.~\ref{fig:ss54aefffnj}.
 \begin{figure}[t!]
\centering
\includegraphics[scale=.28]{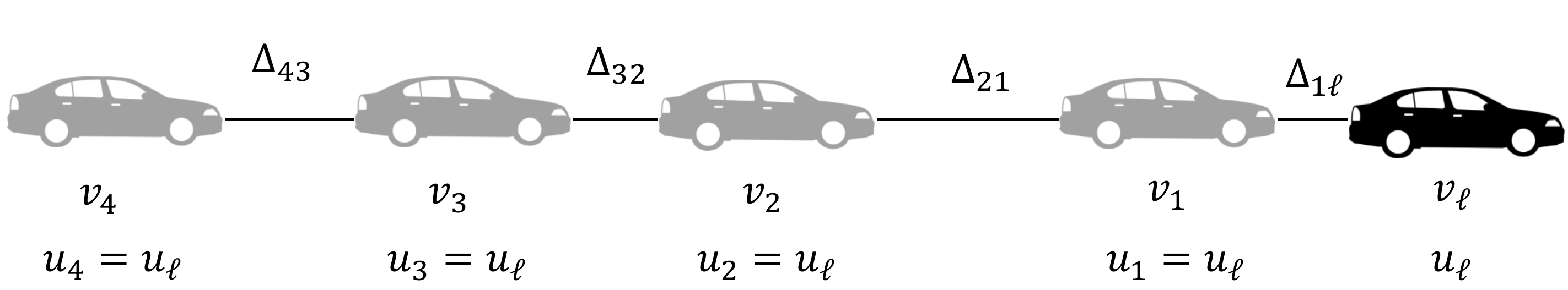}
\caption{Desired inter-vehicular distances $\Delta_{ij}$ and velocity $u_{\ell}$ in a cooperative adaptive cruise control strategy.}
\label{fig:ss54aefffnj}
\end{figure}
 Considering the fact that each vehicle $v_i$ has access to its own position,  the positions of its neighboring vehicles,  and the desired inter-vehicular distances $\Delta_{ij}$, the dynamics of  vehicle $v_i$ can be expressed as  \cite{hao2013stability}
\begin{align}
\ddot{p}_i(t)=\sum_{j\in \mathcal{N}_i}k_p \left (p_j(t)-p_i(t)+\Delta_{ij} \right )\nonumber\\
+k_u \left (u_j(t)-u_i(t)\right)+w_i(t),
\label{eqn:singlle}
\end{align}
where $k_p,k_u>0$ are control gains and $w_i(t)$ models a possible signal attack. The dimension of the attack signal is the same as the dimension of acceleration. This physically means that the attacker, in addition to the defined feedback protocol, applies an additive traction force $F_i(t)=M_iw_i(t)$ to vehicle $v_i$, where $M_i$ is the mass of $v_i$. Dynamics \eqref{eqn:singlle} in matrix form become
\begin{align}
\dot{\boldsymbol{x}}(t)&=\underbrace{\begin{bmatrix}
       \mathbf{0}_{n} & I_{n}         \\[0.3em]
     -k_pL_g & -k_uL_g
     \end{bmatrix}}_{A}{\boldsymbol{x}}(t)+\underbrace{\begin{bmatrix}
       \mathbf{0}_{n\times 1}        \\[0.3em]
    k_p\mathbf{\Delta}   
     \end{bmatrix}}_{B}+\underbrace{\begin{bmatrix}
       \mathbf{0}_{n}        \\[0.3em]
    B   
     \end{bmatrix}}_{F}\mathbf{w}(t),\nonumber\\
     \boldsymbol y(t)&=[ \mathbf{0}_{n} \quad C]\boldsymbol x(t)
\label{eqn:doub}
\end{align}
 where $\boldsymbol{x}=[\mathbf{P} \hspace{2mm} \dot{\mathbf{P}}]^{\sf T}=[{{p}}_1, {{p}}_2, ..., {{p}}_n, \dot{{{p}}}_1, \dot{{{p}}}_2,  ..., \dot{{{p}}}_n]^{\sf T} $, $\mathbf{\Delta}=[\Delta_1, \Delta_2, ..., \Delta_n]^{\sf T}$ in which $\Delta_i=\sum_{j\in \mathcal{N}_i}\Delta_{ij}$. Here $\mathbf{w}(t)$ is the vector of attacks and $\boldsymbol y(t)$ is the vector of sensor measurements. The reason of choosing such an output is that the vehicle longitudinal velocity is available in real-time through either direct GPS measurements or from the estimation with an acceptable accuracy \cite{PiraniTCST}. In order to find the transfer function from the attack signal to the measurements, we take Laplace transform from the acceleration part, second row of  \eqref{eqn:doub}, assuming zero initial condition, which yields
 \begin{align}
s^2X(s)=-k_pL_gX(s)-sk_uL_gX(s)+BW(s),
 \end{align}
where $X(s)$ and $W(s)$ are Laplace transforms of $\boldsymbol x(t)$ and $\boldsymbol w(t)$, respectively. This results in 
\begin{align}\label{eqn:ansfkdjb}
Y(s)=CX(s)=C\left(\underbrace{s^2I+(sk_u+k_p)L_g}_{\Bar{A}(s)}\right)^{-1}BW(s),
 \end{align}

Note that the system \eqref{eqn:doub} is no longer positive and its $L_2$ gain happens at some nonzero frequency. However, since a vehicle has a large mass and inertia, an informed  attacker will not inject a high frequency attack signal in this system. This is because the attack signal changes the acceleration and it is impossible to change a vehicle's acceleration abruptly due its large mass. As a result, a high frequency attack signal can be easily detected by inspecting the received information from the neighboring vehicles. Therefore, in what follows, we assume that the attack signal is slowly varying in time. Under this assumption, we study the attacker-detector game for \eqref{eqn:doub} when the attacker's objective is to minimize the zero frequency of the transfer function whereas the defender's objective is to maximize this quantity. More formally, the objective function of each player in can be written as 
\begin{align}\label{eqn:ansfsdkdjb}
G(0)=C\Bar{A}(0)^{-1}B=\frac{1}{k_p}CL_g^{-1}B,
 \end{align}
In the following subsections we investigate the existence of NE for the attacker-detector game in vehicular platoons where the inter-vehicular communications can be directed or undirected. Moreover, we discuss the effect of the position of the leading vehicle, i.e., leader placement, on the game payoff.

\subsection{Undirected Platoon}

The following proposition discusses the existence of NE for the attacker-detector game in a platoon of vehicles over an undirected path graph (symmetric interactions).

\begin{proposition}
let $\mathcal{G}_u$ be an undirected path graph, corresponding to a platoon of vehicles, with leader vehicle $v_{\ell}$ in one end of the path. Then, for any $f\geq 1$, the attacker detector game admits at least one Nash equilibrium for game \eqref{eqn:ansfsdkdjb} which happens when the attacker chooses $f$ closest nodes to the leader and the detector chooses farthest $f$ nodes from the leader. 
\end{proposition}

\begin{proof}
Based on a specific structure of $L_g^{-1}$ for this topology, as shown in Fig.~\ref{fig:example} (a), the NE is obtained when the attacker chooses first $f$ columns of $L_g^{-1}$ and the detector chooses last $f$ rows of  $L_g^{-1}$. If we denote such row and column selections by the attacker and detector by $B^*$ and $C^*$, according to Lemma \ref{lem:vanmeigh}, it is easy to verify that 
\begin{align}
\sigma_{\max}(CL_g^{-1}B^*)\leq \sigma_{\max}(C^*L_g^{-1}B^*)\leq \sigma_{\max}(C^*L_g^{-1}B)
\end{align}
where $C$ and $B$ are any combination of $f$ rows and columns of $L_g^{-1}$, respectively. In words, any unilateral deviation of the attacker's decision will result in increasing the elements of $C^*L_g^{-1}B^*$, which in turn results in increasing $\sigma_{\max}(C^*L_g^{-1}B)$ (based on lemma \ref{lem:vanmeigh}). Moreover, if $n\leq 2f$ then any unilateral deviation of the detector's choice decreases $\sigma_{\max}(CL_g^{-1}B^*)$. For $n>2f$ the unilateral change in detector's selection may not change the game payoff (the elements of $CL_g^{-1}B$ remain unchanged) which results in multiple NEs with the same value. 
\begin{figure}[t!]
\centering
\includegraphics[scale=.42]{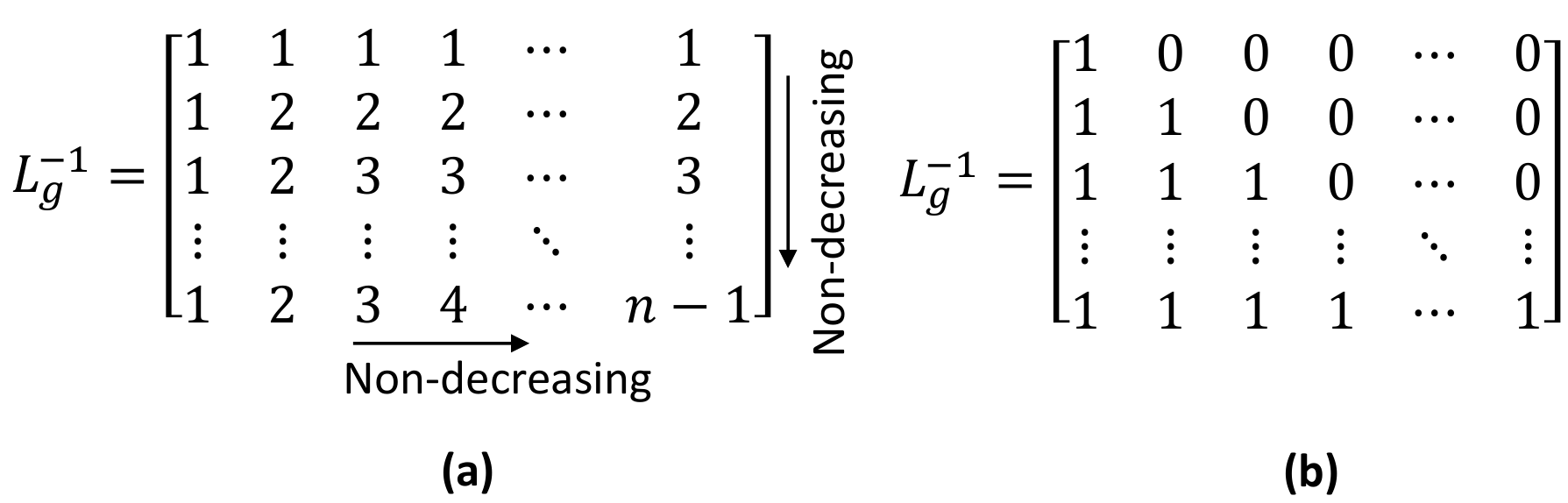}
\caption{Matrix canonical structure for the undirected (a) and directed (b) path graphs.}
\label{fig:example}
\end{figure}
\end{proof}

\begin{figure}[t!]
\centering
\includegraphics[scale=.32]{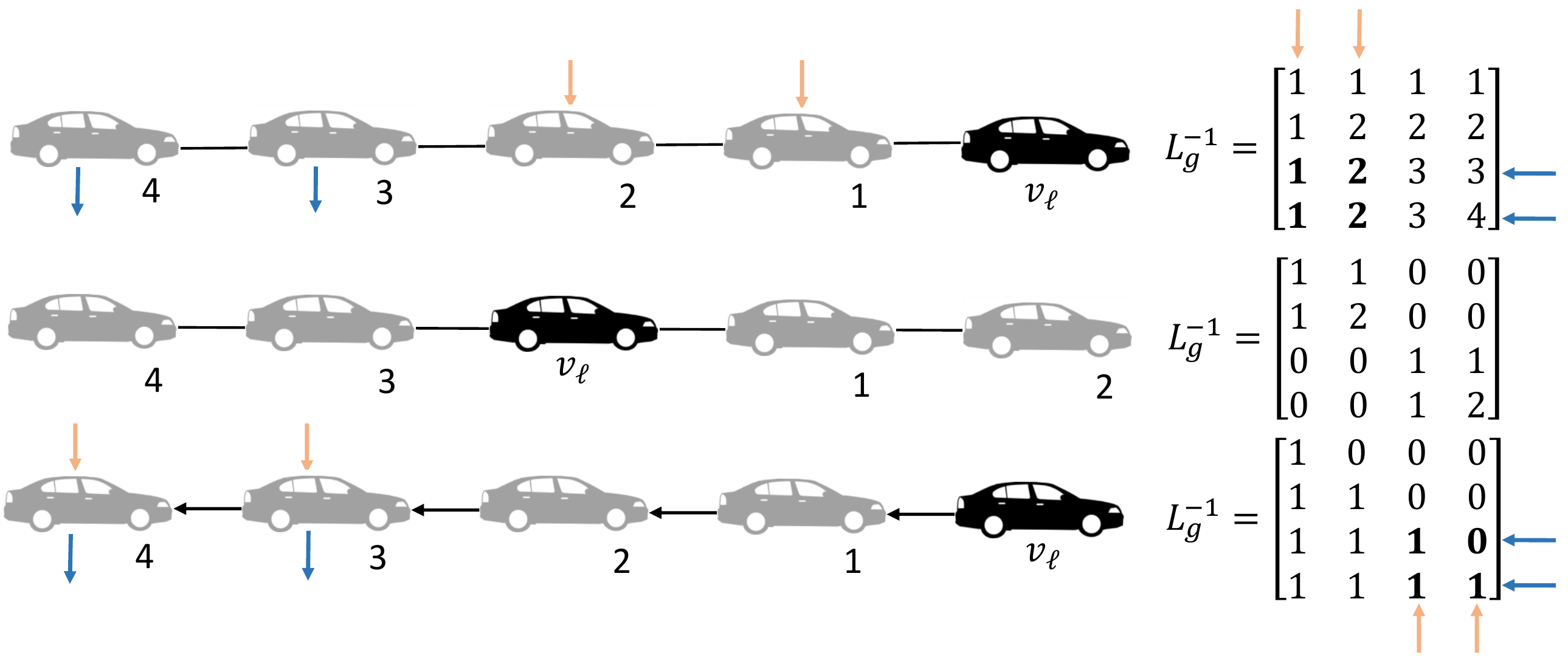}
\caption{A path graph with two leader positions and their corresponding Nash states.}
\label{fig:sssgtdnj}
\end{figure}

\subsubsection{Interplay Between Leader Placement and Security}

Here, we study the impact of the leader's location in a platoon of vehicles on the  value of the attacker-detector game. To study the impact of the leader placement, we use the Stackelberg formulation of the attacker-detector game.  This is because the fact that the attacker-detector game may not admit a NE if the leader is a located at a cut vertex of the underlying connectivity graph (based on Theorem \ref{thm:ssodjvn}).

\begin{theorem}
Consider the Stackelberg attacker-detector game for a platoon of vehicles over an undirected path. Let $J_1$ and $J_2$ denote the corresponding game values when $v_{\ell}$ is not at the head of the platoon  and when $v_{\ell}$ is at the head of the platoon, respectively. Then, we have $J_1 \leq J_2$.
\end{theorem}

\begin{proof}
We label the nodes from one head to the other such that for case 1 (when the leader is in the middle of the platoon), the leader is node $k$ and for case 2 the leader is node $n$.
Define $A_1=CL_{g_1}^{-1}B$ and $A_2=CL_{g_2}^{-1}B$ where $L_{g1}$ and $L_{g2}$ are grounded Laplacian matrices corresponding to case 1 and 2, respectively. For fix and identical matrix $C$ for both cases, we show that $a_{ij}^1\leq a_{ij}^2$, where $a_{ij}^1$ and $a_{ij}^2$ are $ij$-th elements of $A_1$ and $A_2$, respectively. Then for $C^*=\arg\max_C\min_BCL_{g_1}^{-1}B$ which gives the optimal payoff for the max-min game in case $1$, we see that this $C^*$ provides even larger payoff in case 2, which proves the claim.
To prove $a_{ij}^1\leq a_{ij}^2$ for all $i,j=1, 2,..., n$, we use contradiction.  Note that the $i$-th row in matrices $A_1$ and $A_2$ are chosen from the same row index in $L_{g_1}^{-1}$ and $L_{g_2}^{-1}$, since $C$ is fixed to be the same for both cases. Two possibilities: (1) $a_{ij}^1$ is chosen from columns 1 to $k$ in matrix $L_{g_1}^{-1}$. In this case $a_{ij}^1>a_{ij}^2$ (negation of the claim) means that $a_{ij}^1$ is chosen from column $j_1$ and $a_{ij}^2$ is chosen from column $j_2$ where $j_1>j_2$. Thus either there exists a free column (column which is not chosen by the attacker) in $L_{g_1}^{-1}$ in the left hand side of $j_1$ (which contradicts the optimality of the attacker's strategy), or there is no free column in the left side of $j_1$ and since all columns $1, 2,..., j_2$ in $L_{g_2}^{-1}$ are chosen by the optimal attacker, it results in $j_1=j_2$ and since $[L_{g_1}^{-1}]_{ij}\leq [L_{g_2}^{-1}]_{ij}$,  then the initial assumption $a_{ij}^1>a_{ij}^2$ is not true. (2) $a_{ij}^1$ is chosen from columns $k+1$ to $n$. In this case, if $j_1=k+1$, then we have $a_{ij}^1\leq a_{ij}^2$ for all $j_2=1,2, ..., n$ which contradicts the assumption. If $j_1>k+1$, the proof is similar to case (1).
\end{proof}

According to the above theorem, the game value of the attacker-detector game, in a platoon of vehicles, decreases if the leader is moved to an intermediate position in the platoon. We note that it has been shown that changing the leader's position improves the robustness of the platoon dynamics to the communication disturbances \cite{Piranikalle}. However, our results show that a platoon of vehicles will be more secure when the leader is located at the head of the platoon.

\subsection{Directed Platoon}

The role of communication direction on the performance and disturbance rejection in vehicle platooning has been addressed in the literature, under the name of {\it predecessor-following architecture} \cite{hao2013stability}. In this subsection we analyze the security of this platoon topology, which is shown in Fig. \ref{fig:sssgtdnj}, bottom. The following proposition discusses the existence of NE in directed platoons.

\begin{proposition}
Consider the attacker-detector game in a platoon of vehicles over a directed communication graph. Then, for any $f\geq 1$, there exists a Nash equilibrium for the game which belongs to the case where both the attacker and the detector choose farthest $f$  nodes from the leader. 
\end{proposition}
\begin{proof}
 Due to the specific structure of $L_g^{-1}$, which is a lower triangular matrix with all triangle elements equal to 1, Fig.~\ref{fig:example} (b), the NE corresponds to the case where the detector is choosing the last $f$ rows and the attacker chooses last $f$ columns. Then by changing the choices unilaterally by the attacker and detector  and considering Lemma \ref{lem:vanmeigh} the NE is confirmed to exist.
\end{proof}

\begin{remark}\textbf{(Discussion on the Value of NE)}:
It can be easily shown that the game value of the attacker-detector game for the platoon of vehicles in the undirected communication case is higher than that with directed communication.  This shows that undirected platoon is a more secure structure than the directed platoon. Moreover, it can be shown that  the best position of the leader in the undirected platoon from the detector's perspective is the case where the leader is on the head of the platoon. 
\end{remark}

\section{Conclusion}
 \label{sec:conclusion}
An attacker-detector game  on a leader-follower network control system was studied, in which the attacker tries to minimize its visibility and the detector aims to maximize it. The game payoff was the system $L_2$ gain from the attack signal to the measurable outputs. Several conditions for the existence and the value of Nash equilibrium on both directed and undirected trees were studied. Moreover, the problem was studied under the Stackelberg game framework and it was shown that this game can be solved with low computational cost for large scale networks. At the end, these results were applied to vehicular platooning and the optimal network topology and leader position were investigated. A rich avenue for further studies is to extend these results from trees to more general topologies and heterogeneous communication weights.

\begin{appendix}

\subsection{Proof of Lemma \ref{lem:treehinf}}
\label{sec:lem2}

Before proving Lemma  \ref{lem:treehinf} we need some preliminary definitions.

A extension of the above theorem was presented in \cite{Miekkala}. Before that, we have the following definition.
\begin{definition}
 A spanning subgraph of a graph $\mathcal{G}$ is called a 2-tree of $\mathcal{G}$, if and only if, it has two components each of which is a tree. In other words, a 2-tree of $\mathcal{G}$ consists of two trees with disjoint vertices which together span $\mathcal{G}$.  One (or both) of the components may
consist of an isolated node. We refer to $t_{ab,cd}$ as a 2-tree where vertices $a$ and $b$ are in one component
of the 2-tree, and vertices $c$ and $d$ in the other. 
\end{definition}

Based on the above definition, we prove Lemma  \ref{lem:treehinf}.

\begin{proof}
From \cite{Chenapplied} we know that any first order cofactor (principal minor) of the Laplacian matrix $L$ is equal to the number of different spanning trees of the connected graph $\mathcal{G}$. Moreover, from \cite{Miekkala} we know that the second order cofactor ${\rm cof}(L)_{ij, \ell, \ell}$ of the Laplacian matrix $L$ is the
number of different 2-trees $t_{ij,\ell \ell}$ in the connected graph $\mathcal{G}$.   We  know that  $[L_g^
{-1}]_{ij}=\frac{{\rm cof}(L)_{ij, \ell, \ell}}{{\rm det}(L_g)}$. and since $\mathcal{G}$ is a tree (with one spanning tree) we have ${\rm det}(L_g)=1$ which yields $[L_g^
{-1}]_{ij}={\rm cof}(L)_{ij, \ell, \ell}$. Moreover, in $\mathcal{G}$ as a tree, the number of  2-trees  $t_{ij,\ell \ell}$ is equal to the number of trees which contain $v_i$ and $v_j$ and do not contain $v_{\ell}$ and that is equal to
$|\mathcal{P}_{i\ell}\cap \mathcal{P}_{j\ell}|$ which proves the claim.
\end{proof}

\subsection{Proof of Lemma \ref{lem:dirtreehinf}}
\label{sec:lem3}

\begin{proof}
Let $L_{g_d}$ and $L_{g_u}$ be grounded Laplacian matrices of a directed tree and its undirected counterpart, respectively. The proof is based on the fact that for a directed tree with one leader node $v_{\ell}$ we have $L_{g_d}^TL_{g_d}=L_{g_u}$ (proved in \cite{Piranikalle}) which results in $L_{g_d}^{-1}L_{g_d}^{-T}=L_{g_u}^{-1}$. Based on Lemma \ref{lem:treehinf}, we have $[L_{g_u}^{-1}]_{ij}=|\mathcal{P}_{i\ell}\cap \mathcal{P}_{j\ell}|$ which gives
\begin{equation}\label{eqn:aofisa}
[L_{g_u}^{-1}]_{ij}=|\mathcal{P}_{i\ell}\cap \mathcal{P}_{j\ell}|=[L_{g_d}^{-1}]_i[L_{g_d}^{-1}]_j^T
\end{equation}
where $[L_{g_d}^{-1}]_i$ is the $i$-th row of $L_{g_d}^{-1}$. Now consider another node $v_k$ in $\mathcal{G}$. If there is a directed path from $v_k$ to $v_i$ for some $v_k \in \mathcal{V}$, we set the $k$-th element of $[L_{g_d}^{-1}]_i$ equal to 1 and zero otherwise and doing the same work for row $[L_{g_d}^{-1}]_j$. If $v_k \in \mathcal{P}_{i\ell}\cap \mathcal{P}_{j\ell}$ in the undirected graph, then the $k$-th elements of both  $[L_{g_d}^{-1}]_i$ and $[L_{g_d}^{-1}]_j$ are 1 and likewise if we consider all elements of $\mathcal{P}_{i\ell}\cap \mathcal{P}_{j\ell}$, then equality \eqref{eqn:aofisa} will be satisfied and since it should hold for all $i,j=1, 2,..., n-1$, this solution will be unique. 
\end{proof}

\end{appendix}
\bibliographystyle{plain}
\bibliography{biblio}
\end{document}